\documentclass[11pt]{amsart}
\usepackage[babel]{csquotes}
\usepackage{enumitem}
\usepackage{amsmath,amsthm,amssymb,mathrsfs,amsfonts,verbatim,enumitem,color,leftidx,mathtools}
\usepackage{mathabx}
\usepackage{mathdots}
\usepackage{etoolbox} 
\usepackage{bbm}
\usepackage[all,tips]{xy}
\usepackage{graphicx,ifpdf}
\usepackage{stmaryrd}
\usepackage{amssymb}
\usepackage{dsfont}

\ifpdf
\fi
\usepackage[colorlinks]{hyperref}
\hypersetup{
linkcolor=blue,          
citecolor=green,        
}

\newtheorem{thm}{Theorem}[section]
\newtheorem{lem}[thm]{Lemma}
\newtheorem{cor}[thm]{Corollary}
\newtheorem{prop}[thm]{Proposition}

\theoremstyle{definition}
\newtheorem{defi}[thm]{Definition}

\newtheorem{remark}[thm]{Remark}

\theoremstyle{remark}

\numberwithin{equation}{section}

\definecolor{esperance}{rgb}{0.0,0.5,0.0}

\newcommand{\ba}{\mathbf{a}}

\newcommand{\bc}{\mathbf{c}}

\newcommand{\be}{\mathbf{e}}

\newcommand{\bu}{\mathbf{u}}
\newcommand{\bv}{\mathbf{v}}
\newcommand{\bw}{\mathbf{w}}
\newcommand{\bx}{\mathbf{x}}
\newcommand{\by}{\mathbf{y}}
\newcommand{\bz}{\mathbf{z}}

\newcommand{\ortho}[1]{\mathrm{O}_{#1} (K_\nu)}
\newcommand{\sortho}[1]{\SO_{#1} (K_\nu)}

\DeclareMathOperator{\Card}{Card}

\newcommand{\eps}{\epsilon}

\newcommand{\cO}{\mathcal{O}}

\newcommand{\bP}{\mathbb{P}}
\newcommand{\bR}{\mathbb{R}}
\newcommand{\bZ}{\mathbb{Z}}

\newcommand{\bF}{\mathbb{F}}

\newcommand{\bN}{\mathbb{N}}

\newcommand{\SL}{\operatorname{SL}}
\newcommand{\SO}{\operatorname{SO}}

\newcommand{\GL}{\operatorname{GL}}

\newcommand\mb[1]{\mathbf{#1}}

\newcommand{\onto}{\xymatrix{\ar@{>>}[r]&}}

\newcommand{\eq}[1]
{
\begin{equation*}
{#1}
\end{equation*}
}
\newcommand{\eqlabel}[2]
{
\begin{equation}
{#2}\label{#1}
\end{equation}
}

\makeatletter
\newcommand*{\rom}[1]{\expandafter\@slowromancap\romannumeral #1@}
\makeatother

\begin{document}

\title[The submodularity of the covolume function in function fields]{The submodularity of the covolume function in global function fields}

\date{}

\author{Gukyeong Bang}
\address{Gukyeong ~Bang. Department of Mathematical Sciences, Seoul National University, {\it gukyeong.bang@snu.ac.kr}}

\thanks{}

\keywords{}

\def\thefootnote{}
\footnote{{\textbf{Keywords:} function fields, lattices, submodularity. \it 2020 Mathematics
Subject Classification}: Primary 37P15 ; Secondary 11J13.}   
\def\thefootnote{\arabic{footnote}}
\setcounter{footnote}{0}

\begin{abstract}
In this paper, we study the submodularity of the covolume function in global function fields. The submodular property is often needed in the study of homogeneous dynamics, especially to define a Margulis function. We prove that the covolume function is submodular when the class group of the global function field is trivial. 
\end{abstract}

\maketitle

\section{Introduction}
Let $\mathcal{P}(V)$ be a collection of subspaces of a vector space $V$. A function $\phi :\mathcal{P}(V) \rightarrow \bR_{\geq 0}$ is called \textit{submodular} if
\eq{
	\phi(W_1) + \phi(W_2) \geq \phi(W_1 + W_2 ) + \phi(W_1 \cap W_2)
} for any $W_1, W_2 \in \mathcal{P}(V)$.
The submodular property of the covolume function often appears in the study of homogeneous dynamics. 
For example, the submodularity is necessary to construct a \textit{Margulis function}, which guarantees that most of the points in $\SL_n(\bR) / \SL_n(\bZ)$ tend not to be close to the cusp on average (see \cite{EMM}). 
Another example is the \textit{Harder-Narasimhan filtration}, which is a critical tool in the study of parametric geometry (see \cite{HN} and \cite{Sol}).
Hence, proving the submodularity in global function fields will be the first step for several problems of homogeneous dynamics, Teichmüller dynamics, and Diophantine approximations.

Let us first introduce the function field setting we will consider. 
Let $K$ be a global function field over a finite field $\bF_q$ of $q$ elements for a prime power $q$, that is, the function field of a geometrically connected smooth projective curve $\mb{C}$ over $\bF_q$. 
We fix a discrete valuation $\nu$ on $K$ and denote by $K_\nu$ and $\cO_\nu$ the completion of $K$ with respect to $\nu$ and its valuation ring, respectively. 
Fix an element $\pi_\nu\in K$ such that $\nu(\pi_\nu)=1$ (called a \textit{uniformizer}) and denote the residual field by $k_\nu = \cO_\nu /\pi_\nu \cO_\nu$ and $q_\nu = \Card{(k_\nu)}$. 
The absolute value $|\cdot|$ associated with $\nu$ is defined by $|x|=q_\nu^{-\nu(x)}$ for any $x\in K_\nu$. 
Note that $K_\nu$ is isomorphic to $k_\nu(( \pi_\nu ))$, the formal Laurent series with the variable $\pi_\nu$, and the absolute value of $\sum_{i=-\ell}^{\infty} a_i \pi_\nu^i$ is $q_\nu^\ell$ (see \cite{BPP} for instance). 

Let $R_\nu$ be the affine algebra of affine curve $\mb{C} - \{\nu\}$. Note that $R_\nu$ is discrete in $K_\nu$ and plays a similar role to the set of integers $\bZ$ in $\bR$.
We remark that $R_\nu$ is a Dedekind domain in general, but \textit{we assume that $R_\nu$ is a principal ideal domain}, that is, the ideal class group of $R_\nu$ is trivial (see \cite[II.2]{Ser} and \cite[Chapter 1]{Nar}). The standard example is $\mb{C} = \bP^1$ with the standard point at infinity $\nu = [1 : 0]$. In this case, we have $K = \bF_q (T)$, $K_\nu = \bF_q ((T^{-1}))$, $R_\nu = \bF_q [T]$, and $k_\nu = \bF_q$ with the uniformizer $T^{-1}$.

We define the max norm on the exterior algebra of $K_\nu^n$. Let $\be_1, \dots, \be_n$ be the standard basis of $K_\nu^n$.
Denote $\llbracket d_1, d_2 \rrbracket = \{d_1, d_1+1, d_1+2, \dots, d_2\}$ for each $d_1, d_2\in \bZ$ satisfying $d_1 \leq d_2$.
Given $i \in \llbracket 1, n \rrbracket$, we define
\eq{
	\wp^{n}_{i} = \{I \in \llbracket 1, n \rrbracket : \Card{(I)} = i \}.
} Given $I = \{\iota_1, \dots, \iota_i\} \in \wp^{n}_{i}$ and $\bv_{\iota_1}, \dots, \bv_{\iota_i} \in K_\nu^n$, denote $\bv_I = \bv_{\iota_1} \wedge \dots \wedge \bv_{\iota_i}$.
As $\{ \be_I \}_{I \in \wp^{n}_{i}}$ is the standard basis of $\bigwedge^i K_\nu^n$, define the norm on $\bigwedge^{i} K_\nu^n$ by
\eq{
	\| \bv \| = \max_{I \in \wp^{n}_{i}}|c_I|
} given $\bv = \sum_{I \in \wp^{n}_{i}}c_I \be_I \in \bigwedge^{i} K_\nu^n$.

Let $n \in \bN$. Let $V$ be a subspace of $K_\nu^{n}$ and $\Delta$ be a discrete $R_\nu$-module contained in $V$.
We say that $\beta$ is a \textit{$R_\nu$-basis} of $\Delta$ if $\beta$ is a basis of $\Delta$ over $R_\nu$ and is a linearly independent set over $K_\nu$.
The \textit{rank} rk($\Delta$) is defined by the dimension of $\text{span}_{K_\nu}(\Delta)$ as a vector space over $K_\nu$. 
$\Delta$ is called a \textit{lattice} in $V$ if $\Delta$ has a $R_\nu$-basis and $\textrm{rk}(\Delta) = \dim_{K_\nu}(V)$. We will skip writing the containing vector space if it is not essential.
\begin{defi} 
	Let $V$ be a subspace of $K_{\nu}^{n}$ and $\Delta$ be a lattice in $V$. 
A subspace $W$ of $V$ is called $\Delta$-\textit{rational} if $W \cap \Delta $ is a lattice in $W$.
\end{defi}
\begin{defi}
	Let $V$ be a subspace of $K_{\nu}^{n}$ and $\Delta$ be a lattice in $V$. 
	Given a $\Delta$-rational subspace $W$ of $V$,  define 
	\eq{
		d_{\Delta}(W) = \|\bw_1 \wedge \cdots \wedge \bw_i\|,
	} where $\{\bw_1, \dots, \bw_i\}$ is any $R_\nu$-basis of $W \cap \Delta$ and $d_{\Delta}(\{0\}) := 1$.
\end{defi}
Denote by $M_{m,n}(D)$ the set of $m \times n$ matrices over a ring $D$.
\begin{remark}
	$d_\Delta$ is well-defined: Let $\{\bw_1, \dots, \bw_i\}$ and $\{\bw'_1, \dots, \bw'_i\}$ be $R_\nu$-bases of $W \cap \Delta$.
	Denote by $M_1$ and $M_2$ matrices with columns $\bw_1, \dots, \bw_i$ and $\bw'_1, \dots, \bw'_i$, respectively.
	Then there exist $A, B \in M_{i,i}(R_\nu)$ such that $M_1=M_2 A$ and $M_2=M_1 B$, which implies $AB=BA=1$. Thus, we have $\det{A} \in R_\nu^\times$, which implies $|\det{A}| = 1$ (\cite[Section 14]{BPP}). 
	Hence, $\|\bigwedge_{j=1}^{i} \bw_j\| = |\det{A}|\|\bigwedge_{j=1}^{i} \bw'_j\| = \|\bigwedge_{j=1}^{i} \bw'_j\|$.
\end{remark}
In the real case, $d_\Delta(W)$ is the covolume of $W \cap \Delta$. However, there is a normalization issue in the case of function fields. If we define the covolume of a lattice $\Lambda$ in $K_\nu^n$ by $\mathrm{Covol}_n\,(\Lambda) = \mathrm{Vol}_n\,(K_\nu^n / \Lambda)$, where $\mathrm{Vol}_n$ is the Haar measure on $K_\nu^n$, then it is known that
\eq{
\mathrm{Covol}_n\,(R_\nu^n) = q^{(g-1)n},
} where $g$ is the genus of curve $\mb{C}$ (see \cite[Lemma 14.4]{BPP}). Thus, it can be easily deduced that for any lattice $\Delta$ in $K_\nu^n$, we have
\eq{
	\mathrm{Covol}_n\,(\Delta) = \mathrm{Covol}_n\,(R_\nu^n) d_{\Delta}(K_\nu^n).
}
From the above equality, it is natural to define the \textit{covolume} of a lattice $\Lambda$ in $V (\subseteq K_\nu^n)$ by
\eq{
	\mathrm{Covol}\,(\Lambda) = \mathrm{Covol}_{\dim(V)}\,(R_\nu^{\dim(V)}) d_{\Lambda}(V) = q^{(g-1) \dim(V)} d_{\Lambda}(V).
}

The main theorem is the submodularity of $\log{\mathrm{Covol}}$, which is an analogue of \cite[Lemma 5.6]{EMM} for global function fields.
\begin{thm}\label{emm_lem_5.6}
	Assume that $R_\nu$ is a principal ideal domain. Let $\Delta$ be a lattice in $K_{\nu}^{n}$. Then for any $\Delta$-rational subspaces $L$ and $M$ of $K_\nu^n$,
	\eq{
		d_{\Delta}(L)d_{\Delta}(M) \geq d_{\Delta}(L \cap M)d_{\Delta}(L + M),
	} or equivalently,
	\eq{
	\quad \mathrm{Covol}\,(L \cap \Delta) \mathrm{Covol}\,(M \cap \Delta) \geq \mathrm{Covol}\,(L \cap M \cap \Delta) \mathrm{Covol}\,((L+M) \cap \Delta).
	}
\end{thm}
We remark that the above inequalities are equivalent since
\eq{
	\begin{split}
	\mathrm{Covol}\,(R_\nu^{\dim(L)}) \mathrm{Covol}\,(R_\nu^{\dim(M)}) 
	&= q^{(g-1)(\dim(L) + \dim(M))} \\
	&= \mathrm{Covol}\,(R_\nu^{\dim(L+M)}) \mathrm{Covol}\,(R_\nu^{\dim(L \cap M)}).
	\end{split}
}

The submodular property is necessary for several theorems. 
For instance, it is used for the key inequality to define a Margulis function by Kadyrov, Kleinbock, Lindenstrauss and Margulis in the real case:
given $t \in \bR$ and $s \in M_{m,n}(\bR)$, denote 
\eq{
	g_t = \left(\begin{matrix} e^{\frac{t}{m}} I_m & \\ & e^{\frac{-t}{n}} I_n \end{matrix} \right) \quad\text{and}\quad 
	u_s = \left(\begin{matrix} I_m & s \\ & I_n \end{matrix} \right).
} Then there exists $c_0 > 0$ such that given any $t \geq 1$, one can choose $\omega > 0$ such that 
for any lattice $\Delta \in \SL_n(\bR) / \SL_n(\bZ)$ and $i \in \llbracket 1, m+n-1 \rrbracket$,
\eq{
\begin{split}
	\int_{M_{m,n}(\bR)} \alpha_i (g_t u_s \Delta) d \rho(s) &\leq c_0 t e^{-mnt} \alpha_i (\Delta)^{\beta_i} \\
	&+ \omega^{2 \beta_i} \max_{0 < j \leq \min\{ m+n-i, i \}} (\sqrt{\alpha_{i+j}(\Delta) \alpha_{i-j}(\Delta) })^{\beta_i},
\end{split}
} 
where $\alpha_i(\Delta) = \max\{d_{\Delta}(H)^{-1} : \dim(H)=i \}$, $\beta_i = \frac{m}{i}$ for $i \leq m$ and $\beta_i = \frac{n}{m+n-i}$ otherwise, and $\rho$ is the Gaussian distribution.
This inequality explains that when we have an $i$-dimensional subspace $H$ with small $d_{\Delta}(H)$, 
then the $u_s$-average of $\alpha_i (g_t u_s \Delta)$ is bounded by the exponential speed. 
There may be no subspace for dimension $i$, 
but we can still bound it by the second term obtained by the submodularity.
It is known that the above inequality implies that there exists a height function $\tilde{\alpha} = \sum_{i=0}^{m+n} \omega_i \alpha_i$ such that
\eq{
	\int_{M_{m,n}(\bR)} \tilde{\alpha}(g_t u_s \Delta) d \rho(s) \leq 3 c_0 t e^{-mnt} \tilde{\alpha}(\Delta)
} for any $\Delta$ with sufficiently large $\tilde{\alpha}(\Delta)$ depending on $t$. As the height function measures how far a lattice is from the cusp, 
we can bound the trajectory of $g_t u_s \Delta$ using the exponential term (see \cite{KKLM}). In a subsequent paper \cite{BKL}, we obtain an analogous inequality for global function fields and calculate an upper bound of the Hausdorff dimension of singular matrices in function fields.

The submodularity is also used to define the Harder-Narasimhan filtration. 
Given a lattice $\Delta \in \SL_n(\bR) / \SL_n(\bZ)$, let 
\eq{
	S_{\Delta} = \{p_\Gamma := (\text{rk}(\Gamma), \log{\text{Covol} \ (\Gamma)} ) \in \bR^2 : \Gamma \text{ is a sublattice of }\Delta  \}.
} If we denote the extreme points of 
the convex hull of $S_{\Delta}$ by $p_0, \dots, p_k$ ordering by the first coordinate, 
the submodularity implies that for each $i \in \llbracket 0, k \rrbracket$, 
there exists unique $\Delta$-subspace $H_i$ such that $p_i = p_{H_i \cap \Delta}$. 
The filtration $\{0\} = H_0 \cap \Delta < \dots < H_k \cap \Delta = \Delta$ is called the Harder-Narasimhan filtration. 
As we have $\alpha_{\dim(H_i)} (\Delta)= d_\Delta(H_i)^{-1}$ by the extremity, this filtration gives the information of how far $\Delta$ is from the cusp. 
Hence, by observing the filtration of $g_t \Delta$ along time $t$, we can track how often $g_t \Delta$ is close to the cusp in $t \in [0, \infty)$. 
It is proven that there always exists a piecewise linear map approximating the filtration, called \textit{template (or system)}, 
and thus, we can study Diophantine approximations by drawing suitable piecewise linear maps. This approximation theory is called the parametric geometry of numbers (see \cite{Sol}, \cite{DFSU} or \cite{Sax}).

We expect that by Theorem \ref{emm_lem_5.6}, we can prove the analogous theorems above in global function fields.

\noindent \textit{Acknowledgements}. We would like to thank Seonhee Lim and Taehyeong Kim for the helpful discussions. 
This work was supported by the National Research Foundation of Korea (NRF) grant funded by the Korea government (MSIT) (No.2020R1A2C1A01011543).
\section{The Orthogonality in global function fields}
We use the orthogonality for global function fields to prove Theorem \ref{emm_lem_5.6}. 
The orthogonality and corresponding properties in ultrametric space can be found in several references (see \cite[Chapter II]{Wei}, \cite{KlST}, or \cite{PR} for example). 
We organized the properties necessary for our setting with original proofs.
We start by defining the orthogonality in non-archimedean vector spaces using Parseval's identity.
\begin{defi}\label{def_ortho}
Let $V$ be a vector space over $K_\nu$. A subset $S = \{\bv_1, \dots, \bv_i\}$ of $V$ is called \textit{orthogonal} if we have
\eq{
\|{a_1}{\bv_1}+ \dots +{a_i}{\bv_i}\| = \text{max}(|a_1|{\|\bv_1\|}, \dots, |a_i|{\|\bv_i\|})
} for any $a_1, \dots, a_i \in K_\nu$. An orthogonal set $S$ is called \textit{orthonormal} if $\|\bv\| = 1$ for any $\bv \in S$. 
\end{defi}
Note that for each element in the orthogonal group $\mathrm{O}_n(\bR)$, its column vectors form an orthonormal basis in $\bR^{n}$.
A similar group exists in non-archimedean vector spaces. 
Recall that $M_{m,n}(D)$ the set of $m \times n$ matrices whose entries are in a ring $D$.

\begin{defi}
Let $n \in \bN$. The \textit{(non-archimedean) orthogonal group} is defined by
\eq{
	\ \ortho{n} = \{A \in M_{n,n}(\cO_\nu) : | \det{A} | = 1\},
} and the \textit{special orthogonal group} is defined by
\eq{
	\ \sortho{n} = \ortho{n} \cap \SL_n(K_\nu)  = \SL_n(\cO_\nu).
}
\end{defi}
\begin{remark}
$\ortho{n}$ is a group. Given $A \in \ortho{n}$, we have $A^{-1} = (\det{A})^{-1} \mathrm{adj}\,(A)$, where $\mathrm{adj}\,(A)$ is the adjugate matrix of $A$.
Since $A \in M_{n,n}(\cO_\nu)$, we have $\mathrm{adj}\,(A) \in M_{n,n}(\cO_\nu)$ by the ultrametric property. 
Since $|\det{A}| = 1$, we have $A^{-1} \in \ortho{n}$.
\end{remark}
\begin{lem}\label{ortho_group}
For each $g \in \ortho{n}$, $\{g\be_1, \dots, g\be_n\}$ is an orthonormal basis of $K_{\nu}^{n}$.
Conversely, given an orthonormal basis  $\{\bv_1, \dots, \bv_n\}$ of $K_{\nu}^{n}$, the matrix $V$ with columns $\bv_1, \dots, \bv_n$ is in $\ortho{n}$. 
\end{lem}
\begin{proof}
Let $g \in \ortho{n}$ and $\bx = (x_1, \dots, x_n) \in K_{\nu}^{n}$. Then we have
\eqlabel{Eq_Ortho_1}{
	\|g \bx\| = \|\sum_{i=1}^{n} x_i g\be_i\| \leq \max_{1 \leq i \leq n}(\|x_i g\be_i\|) \leq \max_{1 \leq i \leq n}(|x_i|) = \| \bx\|,
}
where the last inequality holds since $g \in M_{n,n}(\cO_\nu)$.
Since $\ortho{n}$ is a group, by applying $g^{-1} \in \ortho{n}$ and $g \bx \in K_{\nu}^{n}$ to (\ref{Eq_Ortho_1}) again, we have $\| \bx\| = \|g^{-1}(g \bx)\| \leq \|g \bx\|$, which implies
\eqlabel{Eq_Ortho_2}{
\|g \bx\| = \| \bx \|
} for any $\bx \in K_{\nu}^{n}$.

Let $\ba = (a_1, \dots, a_n) \in K_\nu^n$. Then
\eq{
	\|\sum_{i=1}^{n} a_i g\be_i\| = \|g \ba\| = \| \ba\| = \max_{1 \leq i \leq n}(\|a_i \be_i\|)
	= \max_{1 \leq i \leq n}(\|a_i g\be_i\|)
} where the second and fourth equalities hold by \eqref{Eq_Ortho_2}. Hence, $\{g\be_1, \dots, g\be_n\}$ is an orthonormal basis of $K_{\nu}^{n}$.

Conversely, let $\beta = \{\bv_1, \dots, \bv_n\}$ be an orthonormal basis of $K_{\nu}^{n}$ and let $V = (v_{ij})$ be the matrix with columns $\bv_1, \dots, \bv_n$.
Then $V \in M_{n,n}(\cO_\nu)$ since $\|\bv_i\| = 1,$ $\forall i \in \llbracket 1, n \rrbracket$.
Observe that for any $M = (m_{ij}) \in M_{n,n}(\cO_\nu)$, we have from the Leibniz formula for the determinant that
\eqlabel{Eq_Ortho_3}{
	|\det{M}| = |\sum_{\sigma \in S_{n}}\text{sgn}(\sigma) \prod_{i=1}^{n} m_{i \sigma(i)} | \leq \max_{\sigma \in S_{n}}|\text{sgn}(\sigma) \prod_{i=1}^{n} m_{i \sigma(i)} | \leq 1.
}
Let $\bx = (x_1, \dots, x_n) \in K_\nu^n$. By the orthogonality of $\beta$, we have
\eq{
	\|V \bx\| 
	= \|x_1 \bv_1 + \dots + x_n \bv_n\|
	= \max(\|x_1 \bv_1\|, \dots, \|x_n \bv_n\|).
} Since $\|\bv_{i}\| = 1, \forall i \in \llbracket 1, n \rrbracket$, we have
\eqlabel{Eq_Ortho_4}{
	\|V \bx\| = \max(|x_1|, \dots, |x_n|) =\| \bx\|.
}
Applying $\bx = V^{-1} \be_i$ to \eqref{Eq_Ortho_4}, we obtain $\|V^{-1} \be_i\| = 1, \forall i \in \llbracket 1, n \rrbracket$, which implies $V^{-1} \in M_{n,n}(\cO_\nu)$. Applying $V, V^{-1} \in M_{n,n}(\cO_\nu)$ to (\ref{Eq_Ortho_3}), we have $|\det{V}| = 1$. Therefore, we have $V \in \ortho{n}$.
\end{proof}
The following lemma says that the non-archimedean norm is invariant under $\ortho{n}$.

\begin{lem}\label{norm_inv}
For any $g \in \ortho{n}$ and $\bv \in \bigwedge K_{\nu}^{n}$, we have
$\|g.\bv\| = \| \bv\|$.
\end{lem}
\begin{proof}
Let $i \in \llbracket 1, n \rrbracket$. 
For any pair $I, J \in \wp^{n}_{i}$, we denote by $\det_{I, J}g$ the ($I, J$)-minor of $g$, that is, the determinant of the $(i, i)$-submatrix of $g$ containing those rows and columns whose indices are in $I, J$, respectively.

Let $\bv = \sum\limits_{I \in \wp^{n}_{i}} c_I \be_I \in \bigwedge^{i}K_{\nu}^{n}$ and $g \in \ortho{n}$. By the ultrametric property,
\eq{
	\|g.\bv\| 
	= \|\sum_{I \in \wp^{n}_{i}} c_I (g.\be_I)\| 
	= \|\sum_{I \in \wp^{n}_{i}} c_I ( \sum_{J \in \wp^{n}_{i}} \mathrm{det}_{I, J}\,g \ \be_J)\|  
	\leq \max_{I, J \in \wp^{n}_{i}} |c_I \mathrm{det}_{I, J}\,g|.
}
Since $|\mathrm{det}_{I, J}\,g| \leq 1$ by $g \in M_{n,n}(\cO_\nu)$ and the ultrametric property, we have
\eqlabel{Eq_Norm_Inv_2}{
	\|g.\bv\| \leq \max_{I, J \in \wp^{n}_{i}} |c_I \mathrm{det}_{I, J}\,g| \leq \max_{I \in \wp^{n}_{i}} |c_I| = \| \bv\|.
}

By applying (\ref{Eq_Norm_Inv_2}) to $g^{-1} \in \ortho{n}$ and $g.\bv \in \bigwedge^{i}K_{\nu}^{n}$ again, 
we obtain $\|\bv\| = \|g^{-1}.(g.\bv)\| \leq \|g.\bv\|$ as well.
\end{proof}
By Lemma \ref{ortho_group} and Lemma \ref{norm_inv}, we have the following corollary.
\begin{cor}\label{onb_norm_inv}
	Let $\beta = \{g_1, \dots, g_n\}$ be an orthonormal basis on $K_{\nu}^{n}$. Let $g$ be the matrix with columns $g_1, \dots, g_n$. Then $\|g.\bv\| = \|\bv\|$ for any $\bv \in \bigwedge K_{\nu}^{n}$.
\end{cor}
There exists an alternative definition of the orthonormal basis, which can be defined for any subspaces of $K_{\nu}^{n}$. We prepare two lemmas to define it.
\begin{lem}\label{basis_to_so}
	Let $i \in \llbracket 1, n-1 \rrbracket$. Given unit vectors $\bv_1, \dots, \bv_i \in K_{\nu}^{n}$ satisfying $\|\bv_1 \wedge \cdots \wedge \bv_i\| = 1$, 
	we can find $\bv_{i+1}, \bv_{i+2}, \dots, \bv_{n} \in K_{\nu}^{n}$ such that the matrix with columns $\bv_1, \dots, \bv_n$ is in $\sortho{n}$.
\end{lem}
\begin{proof}	
Write $\bv_1 \wedge \dots \wedge \bv_i = \sum_{I \in \wp^{n}_{i}}c_I \be_I$.
Since $\max_{I \in \wp^{n}_{i}}{|c_I|} = \|\bigwedge_{j=1}^{i} \bv_j\| = 1$, there exists $I_0 \in \wp^{n}_{i}$ satisfying $|c_{I_0}|=1$. 
Let $I_0^c = \llbracket 1, n \rrbracket - I_0 = \{\iota_1, \dots, \iota_{n-i}\}$. 
Define $\bv_{i+\alpha} = \be_{\iota_\alpha}$ for $1 \leq \alpha < n-i$ and $\bv_{n} = \sigma_{I_0} c_{I_0}^{-1} \be_{\iota_{n-i}}$, where $\sigma_{I_0} \in \{1, -1\}$ is the constant satisfying $\be_{I_0} \wedge \be_{I_0^c} = \sigma_{I_0} \be_{\llbracket 1, n \rrbracket}$. Then
\eqlabel{basis_to_so_1}{
\begin{split}
	\bigwedge_{j=1}^{n}\bv_j &=(\bigwedge_{j=1}^{i}\bv_j) \wedge (\bigwedge_{j=i+1}^{n}\bv_j)
	= (\sum_{I \in \wp^{n}_{i}}c_I \be_I) \wedge (\sigma_{I_0} c_{I_0}^{-1} \be_{I_0^c}) \\
	&= \sigma_{I_0} c_{I_0} c^{-1}_{I_0} \be_{I_0} \wedge \be_{I_0^c}
	= \be_{\llbracket 1, n \rrbracket}.
\end{split}
}
Let $V$ be the matrix with columns $\bv_1, \dots, \bv_n$. Then $\det{V} = 1$ by \eqref{basis_to_so_1}. 
Since $\|\bv_1\|=\cdots=\|\bv_i\| = 1$, $\|\bv_{i+\alpha}\| = \|\be_{\iota_\alpha}\| = 1$ for 
$1 \leq \alpha < n-i$, and $\|\bv_{n}\| = \|\sigma_{I_0} c_{I_0}^{-1} \be_{\iota_{n-i}}\| = |c_{I_0}|^{-1} = 1$, 
we have $V \in M_{n,n}(\cO_\nu)$. Hence, we have $V \in \sortho{n}$.
\end{proof}
\begin{prop}\label{exist_ortho}
Let $i \in \llbracket 1, n \rrbracket$ and $L$ be an $i$-dimensional subspace of $K_{\nu}^{n}$. 
There exists a normalized basis $\{\bv_1, \dots, \bv_i\}$ of $L$ satisfying $\|\bv_1 \wedge \cdots \wedge \bv_i\| = 1$.
\end{prop}
\begin{proof}
We use induction on the dimension of the subspace. When $i = 1$, the normalized vector $\frac{\bv}{\|\bv\|}$ of any non-zero vector $\bv \in L$ is what we want.

Suppose Proposition \ref{exist_ortho} holds for any ($k-1$)-dimensional subspaces of $K_{\nu}^{n}$, and assume that $L$ is a $k$-dimensional subspace of $K_{\nu}^{n}$.
Pick a basis $\{\bv_1, \dots, \bv_k\}$ of $L$ and Let $L'$ be the subspace of $K_{\nu}^{n}$ generated by $\bv_1, \dots, \bv_{k-1}$. 
Then there exists a normalized basis $\{\bw_1, \dots, \bw_{k-1}\}$ of $L'$ satisfying $\|\bw_1 \wedge \cdots \wedge \bw_{k-1}\|=1$ by the induction hypothesis.
By relabeling $\{\bw_1, \dots, \bw_{k-1}, \frac{\bv_k}{\|\bv_k\|}\}$ into new $\{\bv_1, \dots, \bv_k\}$, 
we may assume that there exists a normalized basis $\beta = \{\bv_1, \dots, \bv_k \}$ of $L$ such that $\|\bv_1 \wedge \dots \wedge \bv_{k-1}\|=1$.
Note that since $\bv_1, \dots, \bv_k$ are unit vectors, we have $\|\bv_1 \wedge \cdots \wedge \bv_{k}\| \leq 1$ by the ultrametric property. 
If $\|\bv_1 \wedge \cdots \wedge \bv_{k}\|=1$, $\beta$ satisfies Proposition \ref{exist_ortho}.

Suppose $\|\bv_1 \wedge \cdots \wedge \bv_{k}\| < 1$. Then there exists $l \in \bN$ such that
\eq{
	\|\bv_1 \wedge \cdots \wedge \bv_{k}\|= q_\nu^{-l}.
}
For each $r \in \llbracket 1, k \rrbracket$ and $J \in \wp_r^n$, let $V_r$ be the $(n,r)$-matrix with columns $\bv_1, \dots, \bv_r$, and $V_r^J$ be the $(r, r)$-submatrix of $V_r$ containing the rows whose indices are in $J$.
Since $\|\bv_j\| = 1,  \forall j \in \llbracket 1, r \rrbracket$, we have $|\det{V_r^J} | \leq 1$.
Let $D_{k-1} = \{J \in \wp_{k-1}^n : |\det V_{k-1}^J| = 1  \}$. Note that $D_{k-1}$ is non-empty since $\|\bv_1 \wedge \cdots \wedge \bv_{k-1}\|=1$. We write $V_{k} = (v_{ab})_{1 \leq a \leq n, 1 \leq b \leq k}$.
\vspace{0.3cm}\\ 
\textbf{Claim 1.}\; There exist $I \in D_{k-1}$ and $\iota \in I$ satisfying $|v_{\iota k}| = 1$. 
\begin{proof}[Proof of Claim 1]
Suppose $J \in D_{k-1}$ does not satisfy \textbf{Claim 1}. 
We have $|v_{\iota k}| = 1$ for some $\iota \in \llbracket 1, n \rrbracket - J$ since $\|\bv_k\| = 1$.
Write $J \cup \{\iota\} = \{\iota_1, \dots, \iota_k\}$. Then $\iota_s = \iota$ for some $s \in \llbracket 1, k \rrbracket$.

Applying the Laplace expansion along the $k$-th column to $V_k^{J \cup \{\iota \}} $,
\eqlabel{eq_exist_ortho_2}{
	\det V_k^{J \cup \{\iota \}}
	= \sum_{j \ne s} (-1)^{k+j} v_{\iota_j k} \det V_{k-1}^{J \cup \{\iota \} - \{\iota_j \} } + (-1)^{k+s} v_{\iota k} \det V_{k-1}^J.
} Since $J \in D_{k-1}$ and $|v_{\iota k}| = 1$, it follows from \eqref{eq_exist_ortho_2} that 
\eq{
	1 = |v_{\iota k} \det V_{k-1}^J| 
	\leq \max(|\det V_k^{J \cup \{\iota \}}|,\ |\sum_{j \ne s} (-1)^{k+j} v_{\iota_j k} \det V_{k-1}^{J \cup \{\iota \} - \{\iota_j \}}|).
}
Since $|\det V_k^{J \cup \{ \iota \}}| \leq \|\bv_1 \wedge \cdots \wedge \bv_k \| < 1$ and $\|v_k\| = 1$, we have
\eq{
	1	\leq |\sum_{j \ne s} (-1)^{k+j} v_{\iota_j k} \det  V_{k-1}^{J \cup \{\iota \} - \{\iota_j \}}|
	\leq \max_{j \ne s} (|\det  V_{k-1}^{J \cup \{\iota \} - \{\iota_j \}}|),
}
which implies that there exists $\iota_j \in J - \{\iota \}$ such that $|\det V_{k-1}^{J \cup \{\iota \} - \{\iota_j \}}| = 1$. 
Since $|v_{\iota k}| = 1$, $I := J \cup \{\iota \} - \{\iota_j \} \in D_{k-1}$ satisfies \textbf{Claim 1} with $\iota \in I$.

\end{proof}	

Applying \textbf{Claim 1}, pick $I \in D_{k-1}$ with $|v_{t k }|= 1$ for some $t \in I$. 
\vspace{0.3cm}\\ 
\textbf{Claim 2.}\; there exist $\alpha_{1}, \dots, \alpha_{k-1} \in K_\nu$ which satisfy
\eqlabel{exist_ortho_2}{
	\max{(|\alpha_{1}|, \dots, |\alpha_{k-1}|)} = 1;
}
\eqlabel{exist_ortho_3}{
	\|\alpha_{1}\bv_1 + \alpha_{2}\bv_2 + \cdots + \alpha_{k-1}\bv_{k-1} +\bv_k \| = q_\nu^{-l}.
}
\begin{proof}[Proof of Claim 2]
Since $I \in D_{k-1}$, we have $V_{k-1}^I \in \ortho{k-1}$, which implies that each row, particularly the $t$-th row of $V_{k-1}^I$, has norm one by Lemma \ref{ortho_group} with $(V_{k-1}^I)^{tr}$. 
Hence, there exists $s \in \llbracket 1, k-1 \rrbracket$ such that 
\eqlabel{exist_ortho_5}{
|v_{t s}| = 1.
}
Let $\bv_k^I$ be the $(k-1)$-dimensional vector obtained from $\bv_k$ by eliminating indices in $\llbracket 1, n \rrbracket - I$.
Denote 
\eq{
	\boldsymbol{\alpha}^0 = -(V_{k-1}^I)^{-1} \bv_k^I.
} 
Note that we have $\|\bv_k^I \| = 1$ since $1 = |v_{t k }| \leq \|\bv_k^I \| \leq \|\bv_k\| = 1$. By Lemma \ref{norm_inv}, we have
\eq{
	\|\boldsymbol{\alpha}^0 \| = \|- (V_{k-1}^I)^{-1} \bv_k^I\| = \|-\bv_k^I\| = 1.
}
We define
\eq{
	\boldsymbol{\alpha} = \boldsymbol{\alpha}^0 + \pi_\nu^{l} \be_s.
}
For each $j \in \llbracket 1, k-1 \rrbracket$, let $\alpha_j$ be the $j$-th coordinate of $\boldsymbol{\alpha}$. We will show that $\alpha_1, \dots, \alpha_{k-2}$ and $\alpha_{k-1}$ satisfy \textbf{Claim 2}.

Since $\|\pi_\nu^{l} \be_s\| = q_\nu^{-l} < 1 = \|\boldsymbol{\alpha}^0 \|$, we have $\|\boldsymbol{\alpha}\| = \| \boldsymbol{\alpha}^0 + \pi_\nu^{l}\be_s \| = \|\boldsymbol{\alpha}^0\| = 1$ by the ultrametric property, which satisfies \eqref{exist_ortho_2}.
For \eqref{exist_ortho_3}, write $V_{k-1}^{\{j\}} = V_{k-1}^j$ for each $j \in \llbracket 1, n \rrbracket$
, and define a $(k,k)$-matrix
\eq{
	M_j = \left(\begin{matrix} V_{k-1}^I & \bv^I_{k} \\ V_{k-1}^j & v_{jk} \end{matrix} \right).
}
Note that $M_j$ can be transformed to $V_k^{I \cup \{j\}}$ by a row permutation when $j \notin I$, and $M_j$ has duplicated rows when $j \in I$.
By the determinant formula for block matrices and $|\det V_{k-1}^I| =1$, we have
\eqlabel{exist_ortho_6}{
\begin{split}
	|\det M_j|
	&= | \det V_{k-1}^I \det (v_{jk} - V_{k-1}^j (V_{k-1}^I)^{-1} \bv^I_{k}) | \\
	&= | v_{jk} - V_{k-1}^j (V_{k-1}^I)^{-1} \bv^I_{k} | = | v_{jk} + V_{k-1}^j \boldsymbol{\alpha}^0 |.
\end{split}
} On the other hand, we have
\eqlabel{exist_ortho_7}{
	|\det M_j| = |\det V_k^{I \cup \{j\}}| \leq \|\bv_1 \wedge \cdots \wedge \bv_k \| = q_\nu^{-l}
} for any $j \notin I$ and
\eqlabel{exist_ortho_8}{
	\det M_j = 0
} for any $j \in I$ since $M_j$ has at least two identical rows in this case.

By \eqref{exist_ortho_6}, \eqref{exist_ortho_7} and \eqref{exist_ortho_8}, we have
\eqlabel{exist_ortho_9}{
	| v_{jk} + V_{k-1}^j \boldsymbol{\alpha}^0 | \leq q_\nu^{-l} \quad \text{and}	\quad v_{t k} + V_{k-1}^{t} \boldsymbol{\alpha}^0	= 0
} for any $j \in \llbracket 1, n \rrbracket$.

Hence, for each $j \in \llbracket 1, n \rrbracket$, we have
\eq{
	|\sum_{p=1}^{k-1} \alpha_{p}v_{j p} +v_{j k} | =|V_{k-1}^j \boldsymbol{\alpha}^0 + \pi_\nu^l v_{j s} + v_{jk}| \leq \max(|V_{k-1}^j \boldsymbol{\alpha}^0 + v_{jk}|, |\pi_\nu^l v_{j s} |) \leq q_\nu^{-l},
} where the last inequality holds by \eqref{exist_ortho_9} and $\|\bv_{s}\| = 1$. 
For $j = t$, we have
\eq{
	|\sum_{p=1}^{k-1} \alpha_{p}v_{t p} +v_{t k} | =|V_{k-1}^{t} \boldsymbol{\alpha}^0 + \pi_\nu^l v_{t s} + v_{t k}| = |\pi_\nu^l v_{t s}| = q_\nu^{-l}
}
by \eqref{exist_ortho_5} and \eqref{exist_ortho_9}. Hence, $\alpha_1, \dots, \alpha_{k-2}$, and $\alpha_{k-1}$ satisfy \eqref{exist_ortho_3}, which completes \textbf{Claim 2}.
\end{proof}
Define $\alpha_k = 1$ and $\beta_{j} = \alpha_j \pi_\nu^{-l}$, $\forall j \in \llbracket 1, k \rrbracket$. Let $\bv'_{k} = \sum_{j=1}^{k} \beta_j \bv_j \in L$. 
By \eqref{exist_ortho_2} and \eqref{exist_ortho_3}, we have
$\max{(|\beta_{1}|,  \dots, |\beta_{k}|)} = q_\nu^{l}$ and $\|\bv'_{k}\| = 1$. Moreover,
\eq{
\begin{split}
	\|( \bigwedge_{j=1}^{k-1} \bv_j ) \wedge \bv'_{k} \| &= \|( \bigwedge_{j=1}^{k-1} \bv_j ) \wedge (\sum_{p=1}^{k} \beta_p \bv_p ) \| = \|( \bigwedge_{j=1}^{k-1} \bv_j ) \wedge \beta_k \bv_k \| \\
	&= |\beta_k | \|\bigwedge_{j=1}^{k} \bv_j \| = |\beta_k | q_\nu^{-l} = 1.
\end{split}
} Therefore, $\{\bv_1, \bv_2, \cdots, \bv_{k-1}, \bv'_{k} \}$ is a basis of $L$ satisfying Proposition \ref{exist_ortho}.
	
\end{proof}
By Lemma \ref{basis_to_so} and Proposition \ref{exist_ortho}, we can define the orthonormal basis in linear subspaces of $K_{\nu}^{n}$ in the sense of Hadamard's inequality.
\begin{defi}\label{subsp_orth}
	Let $L$ be a linear subspace of $K_{\nu}^{n}$. A normalized basis $\beta = \{\bv_1, \dots, \bv_i\}$ of $L$ is called \textit{orthonormal} if $\|\bv_1 \wedge \cdots \wedge \bv_i\| = 1$.
\end{defi}

By Lemma \ref{basis_to_so} and Lemma \ref{ortho_group}, the orthonormal basis in Definition \ref{subsp_orth} is an orthonormal set (in the sense of Definition \ref{def_ortho}). 
Conversely, for any orthonormal basis $\beta = \{\bv_1, \dots, \bv_n\}$ of $K_\nu^n$, we have by Lemma \ref{ortho_group} that $\|\bigwedge^{n}_{i=1} \bv_i \| = |\det V| = 1$, where $V$ is the matrix with columns $\bv_1, \dots, \bv_n$. Thus, Definition \ref{subsp_orth} is a generalized definition of the orthonormal basis, which always exists for any subspaces by Proposition \ref{exist_ortho}.

\begin{cor}\label{hadamard}
	Let $i, j \in \llbracket 1, n \rrbracket$ with $i+j \leq n$. For any decomposable pair $\bv \in \bigwedge^{i} K_\nu^n$ and $\bw \in \bigwedge^{j} K_\nu^n$, we have $\| \bv \wedge \bw\| \leq \| \bv\| \| \bw\|$.
\end{cor}
\begin{proof}
Write $\bv = \bv_1 \wedge \dots \wedge \bv_i$ and $\bw = \bw_1 \wedge \dots \wedge \bw_j$. 
If $\bv_1, \dots, \bv_i, \bw_1, \dots, \bw_j $ are linearly dependent over $K_\nu$, 
then it is done since $\bv \wedge \bw = 0$.

Suppose $\bv_1, \dots, \bv_i, \bw_1, \dots, \bw_j $ are linearly independent over $K_\nu$. 
Denote by $L$ and $L'$ the subspaces of $K_{\nu}^{n}$ generated by $\bv_1, \dots, \bv_i$ and $\bw_1, \dots, \bw_j$, respectively. 
By Proposition \ref{exist_ortho}, there exist orthonormal bases $\beta = \{\bx_1, \dots, \bx_i\}$ and $\beta' = \{\by_1, \dots, \by_j\}$ of $L$ and $L'$, respectively. 
Denote $\bx = \bx_1 \wedge \dots \wedge \bx_i$ and $\by = \by_1 \wedge \dots \wedge \by_j$.
Then there exist $c, c' \in K_\nu$ such that
\eqlabel{eq_hadamard_1}{
	\bv = c\bx	\quad \text{and} \quad \bw = c' \by.
} Since $\beta$ and $\beta'$ are orthonormal, we have $\| \bx\| = \| \by\| = 1$, which implies $\| \bv\| = |c|$ and $\| \bw\| = |c'|$ by \eqref{eq_hadamard_1}.
	
On the other hand, since $\beta$ and $\beta'$ are normalized, we have $\|\bx \wedge \by\| \leq 1$ by the ultrametric property.
Hence, it follows from \eqref{eq_hadamard_1} that
\eq{
	\| \bv \wedge \bw\| = \| (c \bx) \wedge (c' \by)\| 
	\leq |c||c'| \|\bx \wedge \by\| 
	\leq |c||c'| 
	= \| \bv\|\| \bw\|.
}
\end{proof}
\begin{cor}
	For each $i \in \llbracket 1, n \rrbracket$, the orthogonal group $\ortho{n}$ acts transitively on the set of decomposable i-vectors up to homothety.
\end{cor}
\begin{proof}
Let $\bv = \bv_1 \wedge \dots \wedge \bv_i$, and $\bw = \bw_1 \wedge \dots \wedge \bw_i$ be decomposable vectors satisfying $\| \bv\| = \| \bw\|$ with
linearly independent sets $\{\bv_1, \dots, \bv_i\}$ and $\{\bw_1, \dots, \bw_i\}$  over $K_\nu$.
Suppose $1 \leq i < n$.
Let $L$ and $L'$ be the subspaces of $K_{\nu}^{n}$ generated by $\bv_1, \dots, \bv_i$, and by $\bw_{1}, \dots, \bw_i$, respectively. 
By Proposition \ref{exist_ortho}, there exist orthonormal bases $\beta = \{\bx_1, \dots, \bx_i\}$ of $L$ and $\beta' = \{\by_1, \dots, \by_{i}\}$ of $L'$. 
Denote $\bx = \bx_1 \wedge \dots \wedge \bx_i$ and $\by = \by_1 \wedge \dots \wedge \by_i$. 
Since $\beta$ and $\beta'$ are bases, there exist $c, c' \in K_\nu - \{0\}$ such that $\bv = c\bx$ and $\bw = c'\by$.
Since $\beta$ and $\beta'$ are orthonormal and $\| \bv\| = \| \bw\|$, we have
 \eq{ 
 	|c| = \|c \bx \| = \| \bv\| = \| \bw \|= \| c' \by\| = |c'|,
} which implies $|c' c^{-1}| = 1$.

On the other hand, there exist $g, g' \in \sortho{n}$ such that $g \be_j = \bx_j$ and $g' \be_j = \by_j,$ $\forall j \in \llbracket 1, i \rrbracket$ by Lemma \ref{basis_to_so}. Thus, we have
\eq{
	c' {c}^{-1} g' g^{-1}  \bv = c' {c}^{-1} g' g^{-1} (c \bx)
	= c' g' {g}^{-1} (g  \be_{\llbracket 1, i \rrbracket}) 
	= c' (g' \be_{\llbracket 1, i \rrbracket})
	= c' \by 
	= \bw.
} Hence, we obtain $\bw = c' {c}^{-1} g' g^{-1}  \bv$ where $|c' {c}^{-1}| = 1$ and $g' g^{-1} \in \sortho{n}$.

When $i = n$, there exist $d, d' \in K_\nu - \{0\}$ such that $\bv = d \be_{\llbracket 1, n \rrbracket}$ and $\bw = d' \be_{\llbracket 1, n \rrbracket}$. 
Since $\| \bv\| = \| \bw\|$, we have $|d^{-1}d'|=1$. Therefore, we have $\bw = d' d^{-1} I_n \bv$, where $I_n$ is the identity matrix in $\sortho{n}$.
\end{proof}

\section{The submodularity of covolumes in global function fields}

The goal of this section is to prove Theorem \ref{emm_lem_5.6}. Given a lattice $\Delta$ in $K_\nu^n$, 
the sublattices induced by $\Delta$-rational subspaces have the following property.
\begin{defi} Let $\Delta$ be a lattice in a subspace $V$ of $K_\nu^n$.
	A $R_\nu$-submodule $\Lambda$ of $\Delta$ is called \textit{primitive} in $\Delta$ if $\text{span}_{K_\nu}(\Lambda) \cap \Delta = \Lambda$. 
	A finite subset $S$ of $\Delta$ is called \textit{primitive} in $\Delta$ if $S$ is a linearly independent set over $K_\nu$ and $\text{span}_{R_\nu}(S)$ is primitive in $\Delta$.
\end{defi}

Note that given a lattice $\Delta$ and a subspace $L$ in $K_{\nu}^{n}$, 
$L \cap \Delta$ is primitive in $\Delta$ since $\text{span}_{K_\nu}(L \cap \Delta) \subseteq L$.

To calculate covolumes in Theorem \ref{emm_lem_5.6}, we need $R_\nu$-bases of 
modules corresponding to the intersection and sum of $\Delta$-rational subspaces.
In Proposition \ref{primitive_extension}, we will show that every submodule of a lattice $\Delta$ in $K_\nu^n$ is a sublattice of $\Delta$, 
and a $R_\nu$-basis of the primitive lattice can be extended to a $R_\nu$-basis of larger lattice when $R_\nu$ is a principal ideal domain.

Let us begin with the equivalent conditions for the discreteness analogous to the real vector spaces.

\begin{lem}\label{lem_is_lattice}
	Let $\Lambda$ be a non-trivial $R_\nu$-submodule of $K_{\nu}^{n}$. Then the following are equivalent:
	\begin{enumerate}
		\item\label{lem_is_lattice_1} $\Lambda$ is discrete.
		\item\label{lem_is_lattice_2} $\lambda_1 (\Lambda) := \inf_{\bx \in \Lambda - \{0\}} \| \bx\| > 0$.
		\item\label{lem_is_lattice_3} $\Card{(\Lambda \cap S)} $ $< \infty$ for any bounded set $S \subseteq K_{\nu}^{n}$.
		\item\label{lem_is_lattice_4} $\Lambda$ contains a shortest non-zero vector.
	\end{enumerate}
\end{lem}

\begin{proof} 
\eqref{lem_is_lattice_1} $\Rightarrow$ \eqref{lem_is_lattice_2}. 
For each $l \in \bZ$, denote by $B(q_\nu^{-l})$ the open ball of radius $q_\nu^{-l}$ centered at 0 in $K_{\nu}^{n}$. 
Since $\Lambda$ is discrete, there exists $l_0 \in \bZ$ such that $B(q_\nu^{-l_0}) \cap \Lambda = \{0\}$. Hence, we have $\lambda_1 (\Lambda) = \inf_{\bx \in \Lambda - \{0\}} \| \bx\| \geq q_\nu^{-l_0} > 0$.

\eqref{lem_is_lattice_2} $\Rightarrow$ \eqref{lem_is_lattice_3}. 
Since the image of $\| \cdot \|$ is $\{q_\nu^{-l} : l \in \bZ \} \cup \{0\}$, there exists $r \in \bZ$ such that $\lambda_1 (\Lambda) = q_\nu^{-r}$. To prove \eqref{lem_is_lattice_3}, it suffices to check the case $S = B(q_\nu^{-l})$, where $l$ is an integer less than $r$.
Suppose $\Lambda \cap S$ is an infinite set by contradiction. 
Then, there exists a sequence $\{\bx_n\}$ of distinct elements in $\Lambda \cap S$. 
Since $S = B(q_\nu^{-l}) = (\pi_\nu^{l} \cO_\nu)^n$ is compact, there exists a converging subsequence $\{\bx_{n_k}\}$ of $\{\bx_n\}$. 
Since $\{\bx_{n_k}\}$ is a Cauchy sequence of distinct elements in $\Lambda \cap S$, for any small $\eps > 0$, there exist $\bx_{n_s}, \bx_{n_t} \in \Lambda \cap S$ such that $0 < \|\bx_{n_s} - \bx_{n_t}\| < \eps$. 
Since $\Lambda$ is an additive group, we have $\bx_{n_s} - \bx_{n_t} \in \Lambda -  \{0\}$. Hence, we obtain
\eq{
	\lambda_1 (\Lambda) = \inf_{\bx \in \Lambda - \{0\}} \| \bx\| \leq \|\bx_{n_s} - \bx_{n_t}\| < \eps,
} which implies $\lambda_1 (\Lambda) = 0$ contradicting \eqref{lem_is_lattice_2}.

\eqref{lem_is_lattice_3} $\Rightarrow$ \eqref{lem_is_lattice_4}. 
There exists a non-zero element $\by \in \Lambda - \{0\}$ since $\Lambda$ is non-trivial. 
Let $l = -\text{log}_{q_\nu}\|\by\| \in \bZ$. 
By \eqref{lem_is_lattice_3} with a bounded set $S=B(q_\nu^{-l + 1})$, $\Lambda \cap S$ is finite. 
Pick a shortest non-zero vector in $\Lambda \cap S$. By the definition of $S$, this vector is also the shortest in $\Lambda - \{0\}$.

\eqref{lem_is_lattice_4} $\Rightarrow$ \eqref{lem_is_lattice_1}. 
By \eqref{lem_is_lattice_4}, there exists a shortest vector $\bv_0$ in $\Lambda - \{0\}$, which implies $\lambda_1 (\Lambda) = \| \bv_0\| > 0$. 
To prove the discreteness of $\Lambda$, it suffices to show $(\bx + B(\lambda_1 (\Lambda))) \cap \Lambda = \{ \bx\}$ for any $\bx \in \Lambda$ since $\bx + B(\lambda_1 (\Lambda))$ is an open neighborhood of $\bx$.

Let $\bx \in \Lambda$ and suppose there exists $\by \in (\bx + B(\lambda_1 (\Lambda))) \cap \Lambda - \{ \bx\}$ by contradiction. 
Since $\bx, \by \in \Lambda$ and $\bx \ne \by$, we have $\bx- \by \in \Lambda - \{0\}$. 
Since $\by \in \bx + B(\lambda_1 (\Lambda))$, we have $\|\bx - \by\| < \lambda_1 (\Lambda)$, 
which contradicts that $\bv_0$ is a shortest non-zero vector in $\Lambda$.
\end{proof}

We will prove that the image of a lattice under any projection onto a $K_\nu$-subspace deleting some lattice vectors is a lattice again, which is used for the inductive proof of Theorem \ref{emm_lem_5.6}. 
We start by checking its discreteness in Lemma \ref{primitive_extension_1}. 
Given a linearly independent set $\alpha$ over $K_\nu$ and a basis $\beta$ of $K_\nu^n$ extending $\alpha$, denote by $\pi_{\alpha, \beta}$ the projection of $K_\nu^n$ along $\text{span}_{K_\nu} \alpha$ onto $\text{span}_{K_\nu}(\beta - \alpha)$.
\begin{remark}
	Note that while the orthogonal projection was convenient for the proof of the submodularity in $\bR^n$, there is no unique orthogonal projection in non-archimedean vector space, which is why we will use a arbitrary projection, which complicates the notation.
\end{remark}
\begin{lem}\label{primitive_extension_1}
	Let $i,k \in \llbracket 1, n \rrbracket$ with $i \leq k$. 
	Let $\Lambda$ be a lattice in a $k$-dimensional subspace of $K_\nu^n$ and $\bv_1, \dots, \bv_i$ be independent vectors over $K_\nu$ in $\Lambda$. 
	Let $\alpha$ be a basis of $\text{span}_{K_\nu}\{\bv_1, \dots, \bv_i\}$ and $\beta$ be a basis of $K_{\nu}^{n}$ extending $\alpha$. Then $\pi_{\alpha, \beta}(\Lambda)$ is a discrete $R_\nu$-submodule of $K_{\nu}^{n}$.
\end{lem}

\begin{proof}
Let $\bx \in \pi_{\alpha, \beta}(\Lambda)$. Choose $\hat{\bx} \in \Lambda$ such that $\pi_{\alpha, \beta}(\hat{\bx}) = \bx$. 
We have $\hat{\bx} - \bx \in \text{ker}(\pi_{\alpha, \beta}) = \text{span}_{K_\nu}\{\bv_1, \dots, \bv_i\}$ since
\eqlabel{primitive_extension_1_1}{
	\pi_{\alpha, \beta}(\hat{\bx} - \bx) = \pi_{\alpha, \beta}(\hat{\bx}) - \pi_{\alpha, \beta}(\bx) = \pi_{\alpha, \beta}(\hat{\bx}) - {\pi_{\alpha, \beta}}^2 (\hat{\bx}) = 0.
} Write $\hat{\bx} - \bx = a_1 \bv_1 + \cdots + a_i \bv_i$ with some $a_1,\dots,a_i \in K_\nu$.
By \cite[Lemma 14.3]{BPP}, $K_\nu / (R_\nu + \cO_\nu)$ is a finite-dimensional vector space over $\bF_q$. 
Hence, there exists a finite subset $H$ of $K_\nu$ such that every $y \in K_\nu$ can be written as $y = y_h + y_r + y_o$ with $y_h \in H$, $y_r \in R_\nu$, and $y_o \in \cO_\nu$. 
Applying it to $a_1, \dots, a_i$, we have $a_j = a_{j, h} + a_{j, r} + a_{j, o}$ for some $a_{j, h} \in H$, $a_{j, r} \in R_\nu$ and $a_{j, o} \in \cO_\nu, \forall j \in \llbracket 1, i \rrbracket$.
Define a lift function $f_{\alpha, \beta} : \pi_{\alpha, \beta}(\Lambda) \rightarrow \Lambda$ by
\eq{
	f_{\alpha, \beta}(\bx) = \bx + \sum_{j = 1}^{i} (a_j - a_{j, r}) \bv_j.
} Note that $f_{\alpha, \beta}(\bx)$ is not canonical since it depends on the choice of $\hat{\bx}$ and $a_{j, r}$.
To check that $f_{\alpha, \beta}$ is well-defined, observe that
\eq{
	f_{\alpha, \beta}(\bx) = (\bx + \sum_{j = 1}^{i} a_j \bv_j) - \sum_{j = 1}^{i} a_{j, r} \bv_j = \hat{\bx} - \sum_{j = 1}^{i} a_{j, r} \bv_j \in \Lambda
} since $\hat{\bx} \in \Lambda$ and $\sum_{j = 1}^{i} a_{j, r} \bv_j \in \text{span}_{R_\nu}\{\bv_1, \dots, \bv_i\} \subset \Lambda$.
For any $\bx \in \pi_{\alpha, \beta}(\Lambda)$, we have by the ultrametric property and $|a_{j,o}| \leq 1$ that
\eqlabel{primitive_extension_1_2}{
\begin{split}
	\|f_{\alpha, \beta}(\bx) - \bx\| &= \|\sum_{j = 1}^{i}(a_{j, h} + a_{j, o})\bv_j\| 
	\leq \max_{j = 1}^{i} (\max(|a_{j, h}|, |a_{j, o}|) \cdot \| \bv_j \|) \\
	&\leq \max(\max_{h \in H}|h|, 1) \cdot \max_{j = 1}^{i} \| \bv_j \|.
\end{split}
}
Denote $C = q_\nu \cdot \max(\max_{h \in H}|h|, 1) \cdot \max_{j = 1}^{i} \| \bv_j \|$.
For any $\bx \in \pi_{\alpha, \beta}(\Lambda)$,
\eq{
\pi_{\alpha, \beta}(f_{\alpha, \beta}(\bx)) = \pi_{\alpha, \beta}(\bx + \sum_{j = 1}^{i} (a_j - a_{j, r}) \bv_j) = \pi_{\alpha, \beta}(\bx) = \bx
} since $\bv_1, \dots, \bv_i \in \text{ker}(\pi_{\alpha, \beta})$ and \eqref{primitive_extension_1_1}. Hence, $f_{\alpha, \beta}$ is injective.

Let $S$ be a bounded set in $K_{\nu}^{n}$. It follows from \eqref{primitive_extension_1_2} that
\eq{
	f_{\alpha, \beta}(S \cap \pi_{\alpha, \beta}(\Lambda)) \subseteq (S+B(C)) \cap f_{\alpha, \beta}(\pi_{\alpha, \beta}(\Lambda)),
} where $B(C)$ is the ball of radius $C$ centered at 0. By the injectivity of $f_{\alpha, \beta}$,
\eq{
\begin{split}
	\Card{(S \cap \pi_{\alpha, \beta}(\Lambda))}  &= \Card{(f_{\alpha, \beta}(S \cap \pi_{\alpha, \beta}(\Lambda)))}  \\
	&\leq \Card{(( S + B(C) ) \cap f_{\alpha, \beta}(\pi_{\alpha, \beta}(\Lambda)) )} \\
	&\leq \Card{(( S + B(C) ) \cap \Lambda )}.
\end{split}
} Since $S+B(C)$ is bounded and $\Lambda$ is discrete, $\Card{(( S + B(C) ) \cap \Lambda )}$ is finite by Lemma \ref{lem_is_lattice} \eqref{lem_is_lattice_1} $\Rightarrow$ \eqref{lem_is_lattice_3}. By Lemma \ref{lem_is_lattice} \eqref{lem_is_lattice_3} $\Rightarrow$ \eqref{lem_is_lattice_1}, $\pi_{\alpha, \beta}(\Lambda)$ is discrete. 
\end{proof}
\begin{lem}\label{independence_k-r}
	Let $\Delta$ be a lattice in $K_{\nu}^{n}$. 
	If a finite subset of $\Delta$ is a linearly independent set over $R_\nu$, then it is a linearly independent set over $K_\nu$.
\end{lem}
\begin{proof}
Let $A = \{\bv_1,\dots,\bv_i\}$ be a linearly independent subset of $\Delta$ over $R_\nu$.
Since $\Delta$ has a $R_\nu$-basis, which is a linearly independent set over $K_\nu$, there exists $g \in \GL_n(K_\nu)$ such that $\Delta = g R_\nu^n$. Since $A$ is a linearly independent set if and only if $g^{-1}A (\subset R_{\nu}^{n})$ is, it suffices to show when $\Delta = R_{\nu}^{n}$.

We use induction on $i = \Card{(A)}$. It is clear when $i = 1$. 
Suppose that Lemma \ref{independence_k-r} holds for any subset whose cardinality is less than $i$ $(i \geq 2)$. 
Recall that $\bv_I = \bv_{\iota_1} \wedge \cdots \wedge \bv_{\iota_j}$ for each $I = \{\iota_1, \dots, \iota_j\} \in \wp_{j}^{n}$ and $\bv_I$ = 1 when $I$ is empty.
Suppose $\sum_{j=1}^{i} a_j \bv_j = 0$ for some $a_1,\dots, a_i \in K_\nu$. 
Given each distinct pair $\alpha, \beta \in \llbracket 1, i \rrbracket$, multiplicating $\bv_{\llbracket 1, i \rrbracket - \{\alpha, \beta \}}$ to the above equation,
\eqlabel{independence_k-r_1}{
	a_\alpha \bv_{\llbracket 1, i \rrbracket - \{\beta \}} + (-1)^{\alpha + \beta + 1} a_\beta \bv_{\llbracket 1, i \rrbracket - \{\alpha \}} = 0.
}
Since both $A - \{\bv_\alpha\}$ and $A - \{\bv_\beta\}$ are linearly independent sets over $R_\nu$ of the cardinality less than $i$, they are linearly independent over $K_\nu$ by the induction hypothesis, and thus, $\bv_{\llbracket 1, i \rrbracket - \{\alpha \}}$ and $\bv_{\llbracket 1, i \rrbracket - \{\beta \}}$ are non-zero.
Since $A \subset R_\nu^n$, all coefficients of $\bv_{\llbracket 1, i \rrbracket - \{\alpha \}}$ and $\bv_{\llbracket 1, i \rrbracket - \{\beta \}}$ are in $R_\nu$. Thus, \eqref{independence_k-r_1} implies that $\frac{a_{\alpha}}{a_{\beta}}$ or $\frac{a_{\beta}}{a_{\alpha}}$ is in $K = \text{frac}(R_\nu)$ if $a_{\alpha} \ne 0$ or $a_{\beta} \ne 0$. Hence, if there exists a non-zero $a_j \in K_\nu$ for some $j \in \llbracket 1, i \rrbracket$, then
\eq{
		\frac{a_1}{a_j} \bv_1 + \cdots + \frac{a_i}{a_j} \bv_i
} is a linear combination over $K$.
Since the linear independence over $R_\nu$ and $K$ are equivalent, $\frac{a_1}{a_j} \bv_1 + \cdots + \frac{a_i}{a_j} \bv_i = 0$ implies $a_t = 0, \forall t \in \llbracket 1, i \rrbracket -\{j\}$, and thus, $a_j \bv_j = 0$.
But since $\bv_j \ne 0$, we have $a_j = 0$, which is a contradiction. Therefore, $A$ is a linearly independent set over $K_\nu$. 			 
\end{proof}

By Lemma \ref{independence_k-r}, given a finite subset $\beta$ of the lattice $\Delta$ in $K_\nu^n$,
it suffices to check the linear independence over $R_\nu$, rather than over $K_\nu$, for $\beta$ to be a $R_\nu$-basis of $\text{span}_{R_\nu}{\beta}$. Thus, we will use the abbreviation $R_\nu$-basis as a basis.

While the lemmas so far hold for any global function field, 
the following Proposition needs to assume that the class group of $K$ is trivial. 
Recall that $R_\nu$ is a Dedekind domain, and $K$ is its field of fractions (see \cite[Section 14]{BPP}). 
Note that little is known about number fields of class number one except for imaginary quadratic fields or cyclotomic fields, while the function fields of class number one are completely classified with the explicit list of equations (see \cite{MS} or \cite{SS}).
\begin{prop}\label{primitive_extension}
	Assume that $R_\nu$ is a principal ideal domain. Let $\Delta$ be a lattice in $K_\nu^n$ and $\Lambda$ be a $R_\nu$-submodule of $\Delta$. Then
	\begin{enumerate}
		\item \label{primitive_extension_sub1} $\Lambda$ is a lattice in $\text{span}_{K_\nu}\Lambda$.
		\item \label{primitive_extension_sub2} Any primitive set $P$ in $\Lambda$ can be extended to a basis of $\Lambda$.
	\end{enumerate}
\end{prop}

\begin{proof}

\eqref{primitive_extension_sub1}. Since $R_\nu$ is Noetherian and $\Delta$ is finitely generated, its $R_\nu$-submodule $\Lambda$ is finitely generated.
Since $\Lambda$ is torsion-free as a submodule of free module $\Delta$, $\Lambda$ is a free module by the structure theorem for finitely generated modules over a principal ideal domain. Therefore, $\Lambda$ has a basis over $R_\nu$. By Lemma \ref{independence_k-r}, $\Lambda$ is a lattice in $\text{span}_{K_\nu}\Lambda$.

\eqref{primitive_extension_sub2}. Let $\overline{\Lambda} = \Lambda / \text{span}_{R_\nu}P$. 
We checked that $\Lambda$ is finitely generated in the proof of \eqref{primitive_extension_sub1}, and thus, $\overline{\Lambda}$ is finitely generated as a quotient module of $\Lambda$.

Suppose $r \bv \in \text{span}_{R_\nu}P$ for some $r \in R_\nu$ and $\bv \in \Lambda$. 
If $r \ne 0$, then $r^{-1} \in \text{frac}(R_\nu) = K \subseteq K_\nu$, and thus, $\bv = r^{-1} (r \bv) \in \text{span}_{K_\nu}P$. 
Since $P$ is a primitive set in $\Lambda$, we have $\bv \in \Lambda \cap \text{span}_{K_\nu}P = \text{span}_{R_\nu}P$. Thus, $\overline{\Lambda}$ is torsion-free.
Since $\overline{\Lambda}$ is torsion-free and finitely generated over a principal ideal domain $R_\nu$, $\overline{\Lambda}$ is free by the structure theorem. From the exact sequence $0 \rightarrow \text{span}_{R_\nu}P \rightarrow \Lambda \rightarrow \overline{\Lambda} \rightarrow 0$, we have
$\Lambda \cong \text{span}_{R_\nu}P \oplus \overline{\Lambda}$, and we can find an extended basis of $\Lambda$ using the freeness of $\overline{\Lambda}$.
\end{proof}
We remark that any finitely generated torsion-free module of rank $n$ over a Dedekind domain $R$ is isomorphic to the direct sum of $R^{n-1}$ and a projective module $I$ of rank 1. 
We have $I = R$, and thus, given module has a basis when $R$ is a principal ideal domain, but it fails for general Dedekind domains (See \cite[Chapter VII.4, Proposition 24]{Bou}).

\begin{lem}\label{rational_operation}
	Assume that $R_\nu$ is a principal ideal domain. 
	Let $\Delta$ be a lattice in $K_{\nu}^{n}$, and let $L$ and $M$ be $\Delta$-rational subspaces of $K_{\nu}^{n}$. Then $L + M$ and $L \cap M$ are $\Delta$-rational subspaces of $K_{\nu}^{n}$.
\end{lem}

\begin{proof}
	Since $L$ and $M$ are $\Delta$-rational, there exist bases $\beta_L$ of $L\cap \Delta$ and $\beta_M$ of $M \cap \Delta$. 
	Then $\beta_L \cup \beta_M $ is a generating set of $L+M$ in $\Delta$ over $K_\nu$.
	Thus, to prove $L+M$ is $\Delta$-rational, it suffices to prove the following claim: 
	\vspace{0.3cm}\\ 
 \textbf{Claim 1.}\; If a subspace $N$ of $K_{\nu}^{n}$ contains a subset $S_N \subset N \cap \Delta$ satisfying $\text{span}_{K_\nu}S_N = N$, then $N$ is $\Delta$-rational.
 \begin{proof}[Proof of Claim 1] 
 Let $j = \dim (N)$ and $s = \text{rk}(N \cap \Delta)$. 
 Let $\beta_1 = \{\bw_1, \dots, \bw_j\}$ be a basis of $N$ over $K_\nu$ contained in $S_N$ and 
 let $\beta_2 = \{\bu_1, \dots, \bu_s\}$ be a basis of $N \cap \Delta$ by Proposition \ref{primitive_extension} \eqref{primitive_extension_sub1}. 
 Since $\beta_2$ is a linearly independent subset of $N$ over $K_\nu$, we have $s \leq j$.
 Denote by $W$ the $(n, j)$-matrix with columns $\bw_1, \dots, \bw_j$ and $U$ the $(n, s)$-matrix with columns $\bu_1, \dots, \bu_s$. Since $\beta_1 \subset N \cap \Delta$ and $\text{span}_{R_\nu} \beta_2 = N \cap \Delta$, 
 there exists $B \in M_{s,j}(R_\nu)$ such that $W=UB$.
 
 Denote by $\text{rank}_M(X)$ the matrix rank of the matrix $X$. Since $\beta_1$ is a linearly independent set over $K_\nu$, $\text{rank}_M(W) = j$. Hence, we have 
 \eq{
	 j = \text{rank}_M(W) = \text{rank}_M(UB) \leq \text{rank}_M(B) \leq \min(s, j).
 }
 Hence, we have $j \leq s$, and thus, $N$ is $\Delta$-rational.
 \end{proof}
 
Now, let us prove that $L \cap M$ is $\Delta$-rational. Since $L$, $M$, and $L+M$ are $\Delta$-rational and $\dim (L \cap M) = \text{dim}(L) + \text{dim}(M)- \text{dim}(L+M)$, we have
\eqlabel{rational_operation_1}{
\begin{split}
	\text{rk}(L \cap M \cap \Delta) &\leq \text{dim}(L) + \text{dim}(M) - \text{dim}(L+M) \\
	&= \text{rk}(L \cap \Delta) + \text{rk}(M \cap \Delta) - \text{rk}((L+M) \cap \Delta).
\end{split}	
}

On the other hand, define a $R_\nu$-module homomorphism $\phi$ from $L \cap \Delta$ to $\frac{(L+M)\cap \Delta}{M \cap \Delta}$ by $\phi(\bv) = \bv + (M \cap \Delta)$.
Observe that since $\text{span}_{K_\nu}(M \cap \Delta) = M$,
\eq{
	M \cap \Delta = \text{span}_{K_\nu}(M \cap \Delta) \cap \{(L+M) \cap \Delta\},
} which means that $M \cap \Delta$ is primitive in $(L+M) \cap \Delta$.

Let $\beta = \{\bv_1, \dots, \bv_m\}$ be a basis of $M \cap \Delta$. 
$\beta$ is primitive in $(L+M) \cap \Delta$ since $M \cap \Delta$ is. 
By Proposition \ref{primitive_extension} \eqref{primitive_extension_sub2}, $\beta$ can be extended to a basis $\{\bv_1, \dots, \bv_s\}$ of $(L+M) \cap \Delta$. Hence, we have
\eq{
	\frac{(L+M) \cap \Delta}{M \cap \Delta} \cong \bigoplus_{j=m+1}^{s} R_\nu \bv_j,
} which means that $\frac{(L+M) \cap \Delta}{M \cap \Delta}$ is a free $R_\nu$-module, and 
the rank of $\frac{(L+M)\cap \Delta}{M \cap \Delta}$ is $\text{rk}((L+M) \cap \Delta) - \text{rk}(M \cap \Delta)$. Thus, by the rank-nullity theorem for free modules over a principal ideal domain, we have
\eqlabel{rational_operation_2}{
\begin{split}
	\text{rk}(L \cap \Delta) &= \text{rk}(\text{ker}(\phi)) + \text{rk}(\text{Image}(\phi)) \\
	&= \text{rk}(L \cap M \cap \Delta) + \text{rk}(\text{Image}(\phi)) \\
	&\leq \text{rk}(L \cap M \cap \Delta) + \text{rk}(\frac{(L+M)\cap \Delta}{M \cap \Delta}) \\
	&= \text{rk}(L \cap M \cap \Delta) +  \text{rk}((L+M)\cap \Delta) - \text{rk}(M \cap \Delta ).
\end{split}
}
Combining \eqref{rational_operation_1} and \eqref{rational_operation_2} and using the $\Delta$-rationality of $L, M$, and $L+M$,
\eq{
\begin{split}
	\text{rk}(L \cap M \cap \Delta) 
	&= \text{rk}(L \cap \Delta) + \text{rk}(M \cap \Delta) - \text{rk}((L+M) \cap \Delta)  \\
	&= \text{dim}(L) + \text{dim}(M) - \text{dim}(L+M) \\
	&= \text{dim}(L \cap M).
\end{split}
} Therefore, $L \cap M$ is $\Delta$-rational.
\end{proof}

Now, we are ready to prove Theorem \ref{emm_lem_5.6}.
\begin{proof}[proof of Theorem \ref{emm_lem_5.6}]
	
By Lemma \ref{rational_operation}, $L+M$ and $L \cap M$ are $\Delta$-rational. 
Let $i = \text{dim}(L \cap M) = \text{rk}(L \cap M \cap \Delta)$ and $\alpha_0 = \{\bv_1, \dots, \bv_i\}$ be a basis of $L \cap M \cap \Delta$ by Proposition \ref{primitive_extension} \eqref{primitive_extension_sub1}. 
Applying Lemma \ref{basis_to_so} and Proposition \ref{exist_ortho}, take an orthonormal basis $\alpha = \{\bc_1, \dots, \bc_i\}$ of $L \cap M = \text{span}_{K_\nu} \alpha_0$ over $K_{\nu}$ and its extended orthonormal basis $\beta = \{\bc_1, \dots, \bc_n\}$ of $K_{\nu}^{n}$ over $K_{\nu}$. 
By Lemma \ref{primitive_extension_1}, $\pi_{\alpha, \beta}(\Delta)$ is a discrete $R_\nu$-module.

Let $H$ be a $\Delta$-rational subspace in $K_{\nu}^{n}$ containing $L \cap M$. 
Since the kernel of $\pi_{\alpha, \beta}$ is $L \cap M (\subseteq H)$, we have $\pi_{\alpha, \beta}(\Delta) \cap \pi_{\alpha, \beta}(H) = \pi_{\alpha, \beta}(\Delta \cap H)$. Hence,

\eqlabel{emm_lem_5.6_0}{
\begin{split}
	\text{rk}(\pi_{\alpha, \beta}(\Delta) \cap \pi_{\alpha, \beta}(H))
	&= \text{dim}(\text{span}_{K_\nu}\pi_{\alpha, \beta}(\Delta \cap H)  ) \\
	&= \text{dim}(\pi_{\alpha, \beta}(\text{span}_{K_\nu}(\Delta \cap H))  ) \\
	&= \text{dim}(\pi_{\alpha, \beta}(H)  ),
\end{split}
} where the last equality holds since $H$ is $\Delta$-rational. 
\vspace{0.3cm}\\ 
\textbf{Claim 1.}\; $d_{\Delta}(H) = d_{\pi_{\alpha, \beta}(\Delta)}(\pi_{\alpha, \beta}(H))d_{\Delta}(L \cap M)$.
\begin{proof}[Proof of Claim 1] Since $L \cap M \cap \Delta = \text{span}_{R_\nu} \alpha_0$ and $L \cap M = \text{span}_{K_\nu} \alpha_0$, we have $\text{span}_{R_\nu}\alpha_0$ = $\text{span}_{K_\nu}\alpha_0 \cap (H \cap \Delta)$, which implies that $\alpha_0$ is a primitive set in $H \cap \Delta$. 
By Proposition \ref{primitive_extension} \eqref{primitive_extension_sub2}, there exists a basis $\alpha_H = \{\bv_1, \dots, \bv_j\}$ of $H \cap \Delta$ extending $\alpha_0$. Then we have
\eq{
\begin{split}
	d_{\Delta}(L \cap M) &= \|\bv_1 \wedge \cdots \wedge \bv_i  \| = \|\bv_{\llbracket 1, i \rrbracket}  \|; \\
	d_{\Delta}(H) &= \|\bv_1 \wedge \cdots \wedge \bv_j  \| = \|\bv_{\llbracket 1, j \rrbracket}  \|.
\end{split}
}
Note that $\bx - \pi_{\alpha, \beta}(\bx) \in \text{ker}(\pi_{\alpha, \beta})$ = $\text{span}_{K_\nu}\{\bv_1, \dots, \bv_i\}$ for any $\bx \in K_\nu^n$ since $\pi_{\alpha, \beta}$ is a projection. Hence, 
\eqlabel{emm_lem_5.6_1}{
\begin{split}
	d_{\Delta}(H) &= \| \bv_{\llbracket 1, i \rrbracket} \wedge \bigwedge_{s = i+1}^{j}  ( \pi_{\alpha, \beta}(\bv_s) + (\bv_s - \pi_{\alpha, \beta}(\bv_s) ))  \| \\
	&= \| \bv_{\llbracket 1, i \rrbracket} \wedge \bigwedge_{s = i+1}^{j}    \pi_{\alpha, \beta}(\bv_s)   \|.
\end{split}
} Let $\wp = \{J \subset \llbracket i+1, n \rrbracket : \Card{(J)} = j-i\}$.
Since $\bv_1, \dots, \bv_i \in \text{span}_{K_\nu} \alpha$ and $\pi_{\beta}(\bv_{i+1}), \dots$, $\pi_{\beta}(\bv_{j}) \in \text{span}_{K_\nu}(\beta - \alpha)$, using the basis $\beta$, we have
\eqlabel{emm_lem_5.6_2}{
	\bv_{\llbracket 1, i \rrbracket} = a \bc_{\llbracket 1, i \rrbracket},	\ \ 
	\bigwedge_{s = i+1}^{j}    \pi_{\alpha, \beta}(\bv_s) = \sum_{J \in \wp} b_{J} \bc_J
} with some $a, b_J \in K_\nu$ for each $J \in \wp$. It follows from \eqref{emm_lem_5.6_1} and \eqref{emm_lem_5.6_2} that
\eqlabel{emm_lem_5.6_3}{
\begin{split}
	d_{\Delta}(H) 
	&=\| ( a \bc_{\llbracket 1, i \rrbracket}) \wedge (\sum_{J \in \wp} b_{J} \bc_J)  \| 
	=|a| \|(\sum_{J \in \wp} b_{J} \bc_{\llbracket 1, i \rrbracket} \wedge \bc_J)  \| \\
	&=|a| \max_{J \in \wp} |b_J|,
\end{split}
} where the last equality holds by Corollary \ref{onb_norm_inv}.
On the other hand, we have by \eqref{emm_lem_5.6_2} and Corollary \ref{onb_norm_inv} again that
\eqlabel{emm_lem_5.6_4}{
	d_{\Delta}(L \cap M) = \|\bv_{\llbracket 1, i \rrbracket}  \| = \| a \bc_{\llbracket 1, i \rrbracket}\| = |a|.
}

We need to check that $\pi_{\alpha, \beta}(H) \cap \pi_{\alpha, \beta}(\Delta) = \pi_{\alpha, \beta}(H \cap \Delta) $ has a $R_\nu$-basis to calculate $d_{\pi_{\alpha, \beta}(\Delta)}(\pi_{\alpha, \beta}(H))$, 
and we claim that $\pi_{\alpha, \beta}(\alpha_H - \alpha_0)$ is such a basis. 
Suppose that $\sum_{s = i+1}^{j} a_{s} \pi_{\alpha, \beta}(\bv_s) = 0$ for some $a_{i+1}, \dots, a_j \in K_\nu$. 
Then $\sum_{s = i+1}^{j} a_{s} \bv_s $ is an element of $\text{ker}(\pi_{\alpha, \beta}) = \text{span}_{K_\nu} \{\bv_1, \dots, \bv_i\}$. 
Since $\bv_1, \dots, \bv_j$ are linearly independent over $K_\nu$, we have $a_{i+1}=\cdots=a_j = 0$, and thus, $\pi_{\alpha, \beta}(\alpha_H  -  \alpha_0)$ is a linearly independent set over $K_\nu$.

On the other hand, let $\bx \in \pi_{\alpha, \beta}(H \cap \Delta)$. 
Choose $\tilde{\bx} \in H \cap \Delta$ such that $\pi_{\alpha, \beta}(\tilde{\bx}) = \bx$. Since $\alpha_H$ is a basis of $H \cap \Delta$, there exist $b_1, \dots, b_j \in R_\nu$ such that $\tilde{\bx} = \sum_{s=1}^{j} b_s \bv_s$.
Since $\text{ker}(\pi_{\alpha, \beta}) = \text{span}_{K_\nu} \{\bv_1, \dots, \bv_i\}$, we have
\eq{
	\bx = \pi_{\alpha, \beta}(\sum_{s=1}^{j} b_s \bv_s) = \sum_{s=i+1}^{j} b_s \pi_{\alpha, \beta}(\bv_s),
} which implies that $\pi_{\alpha, \beta}(\alpha_H - \alpha_0)$ generates $\pi_{\alpha, \beta}(H \cap \Delta)$ over $R_\nu$. Hence, $\pi_{\alpha, \beta}(\alpha_H - \alpha_0)$ is a $R_\nu$-basis of $\pi_{\alpha, \beta}(H \cap \Delta)$.

In particular, applying $H=K_\nu^n$, we get a $R_\nu$-basis of $\pi_{\alpha, \beta}(\Delta)$, which implies that $\pi_{\alpha, \beta}(\Delta)$ is a lattice in $\pi_{\alpha, \beta}(K_\nu^n)$. 
By \eqref{emm_lem_5.6_0}, $\pi_{\alpha, \beta}(H)$ is $\pi_{\alpha, \beta}(\Delta)$-rational, and thus, we can define $d_{\pi_{\alpha, \beta}(\Delta)}(\pi_{\alpha, \beta}(H))$ by a basis $\pi_{\alpha, \beta}(\alpha_H - \alpha_0)$. 
It follows from \eqref{emm_lem_5.6_2} and Corollary \ref{onb_norm_inv} that
\eqlabel{emm_lem_5.6_5}{
	d_{\pi_{\alpha, \beta}(\Delta)}(\pi_{\alpha, \beta}(H))
	= \|\bigwedge_{s=i+1}^{j} \pi_{\alpha, \beta}(\bv_{s})\|
	=\|\sum_{J \in \wp} b_{J} c_J\| =\max_{J \in \wp} |b_{J}|.
}
Thus, \textbf{Claim 1} follows from \eqref{emm_lem_5.6_3}, \eqref{emm_lem_5.6_4}, and \eqref{emm_lem_5.6_5}.
\end{proof}
Applying $L$, $M$, and $L+M$ to $H$ of \textbf{Claim 1}, it suffices to prove that
\eq{
	d_{\pi_{\alpha, \beta}(\Delta)}(\pi_{\alpha, \beta}(L))d_{\pi_{\alpha, \beta}(\Delta)}(\pi_{\alpha, \beta}(M)) \geq d_{\pi_{\alpha, \beta}(\Delta)}(\pi_{\alpha, \beta}(L+M)).
} Since $L \cap M \subseteq L$ and $L \cap M \subseteq M$, we have
\eq{
	\begin{split}
		\pi_{\alpha, \beta}(L) \cap \pi_{\alpha, \beta}(M) &= \pi_{\alpha, \beta}(L \cap M) = \{0\}; \\
		\pi_{\alpha, \beta}(L) + \pi_{\alpha, \beta}(M) &= \pi_{\alpha, \beta}(L+M),
	\end{split}
} which implies that
\eq{
	\text{dim}(\pi_{\alpha, \beta}(L+M)) = \text{dim}(\pi_{\alpha, \beta}(L)) + \text{dim}(\pi_{\alpha, \beta}(M)).
}
Let $s = \text{dim}(\pi_{\alpha, \beta}(L))$ and $t = \text{dim}(\pi_{\alpha, \beta}(M))$. 
Since $\pi_{\alpha, \beta}(L), \pi_{\alpha, \beta}(M)$, and $\pi_{\alpha, \beta}(L+M)$ are $\pi_{\alpha, \beta}(\Delta)$-rational by \eqref{emm_lem_5.6_0}, their intersection with $\pi_{\alpha, \beta}(\Delta)$ has a basis whose cardinality is the same as the dimension, respectively.
Let $\gamma_L = \{\bx_1, \dots, \bx_s\}$, $\gamma_M = \{\by_1, \dots, \by_t\}$, and $\gamma_{L+M} = \{\bz_1, \dots, \bz_{s+t}\}$ be a basis of $\pi_{\alpha, \beta}(L) \cap \pi_{\alpha, \beta}(\Delta)$, $\pi_{\alpha, \beta}(M) \cap \pi_{\alpha, \beta}(\Delta)$ and $\pi_{\alpha, \beta}(L+M) \cap \pi_{\alpha, \beta}(\Delta)$, respectively.
Then we have
\eqlabel{emm_lem_5.6_6}{
\begin{split}
	d_{\pi_{\alpha, \beta}(\Delta)}(\pi_{\alpha, \beta}(L)) &= \|\bx_{\llbracket 1, s \rrbracket}\|; \\
	d_{\pi_{\alpha, \beta}(\Delta)}(\pi_{\alpha, \beta}(M)) &= \|\by_{\llbracket 1, t \rrbracket}\|; \\
	d_{\pi_{\alpha, \beta}(\Delta)}(\pi_{\alpha, \beta}(L+M)) &= \|\bz_{\llbracket 1, s+t \rrbracket}\|.
\end{split}
} By Corollary \ref{hadamard}, we obtain that
\eqlabel{emm_lem_5.6_7}{
\begin{split}
	\|\bx_{\llbracket 1, s \rrbracket} \wedge \by_{\llbracket 1, t \rrbracket}\|
	&\leq \|\bx_{\llbracket 1, s \rrbracket}\| \cdot \|\by_{\llbracket 1, t \rrbracket}\| \\
	&= d_{\pi_{\alpha, \beta}(\Delta)}(\pi_{\alpha, \beta}(L)) \cdot d_{\pi_{\alpha, \beta}(\Delta)}(\pi_{\alpha, \beta}(M)).
\end{split}
} Since $\pi_{\alpha, \beta}(L+M) \cap \pi_{\alpha, \beta}(\Delta)$ contains $\gamma_L$ and $\gamma_M$, and $\gamma_{L+M}$ is its basis, 
there exists $C = (c_{ij}) \in M_{s+t, s+t}(R_\nu)$ such that
\eqlabel{emm_lem_5.6_8}{
	\bx_i = \sum_{l=1}^{s+t} c_{il}\bz_l \quad \text{and} \quad
	\by_j = \sum_{l=1}^{s+t} c_{(s+j)l}\bz_l  
} for any $i \in \llbracket 1, s \rrbracket$ and $j \in \llbracket 1, t \rrbracket$. By \eqref{emm_lem_5.6_6}, \eqref{emm_lem_5.6_7} and \eqref{emm_lem_5.6_8}, we have
\eqlabel{emm_lem_5.6_10}{
\begin{split}	
	\| \bx_{\llbracket 1, s \rrbracket} \wedge \by_{\llbracket 1, t \rrbracket} \|
		&= |\det C| \cdot \|\bz_{\llbracket 1, s+t \rrbracket} \| = |\det C| \cdot d_{\pi_{\alpha, \beta}(\Delta)}(\pi_{\alpha, \beta}(L+M)) \\
		&\leq d_{\pi_{\alpha, \beta}(\Delta)}(\pi_{\alpha, \beta}(L))  d_{\pi_{\alpha, \beta}(\Delta)}(\pi_{\alpha, \beta}(M)).		
	\end{split}
}
Note that $\det C \in R_\nu$ since $C$ is a matrix over $R_\nu$. 
Since $R_\nu \cap \cO_\nu = \bF_q$ (\cite[Chap 14.2]{BPP}), we have $|\det C| \geq 1$ or $\det C = 0$.
It is immediate that $\gamma_L \cup \gamma_M$ is a linearly independent set over $K_\nu$, and thus, $\det C \ne 0$ from \eqref{emm_lem_5.6_8}.
Hence, it follows from \eqref{emm_lem_5.6_10} that
\eq{
	d_{\pi_{\alpha, \beta}(\Delta)}(\pi_{\alpha, \beta}(L)) \cdot d_{\pi_{\alpha, \beta}(\Delta)}(\pi_{\alpha, \beta}(M))
	\geq d_{\pi_{\alpha, \beta}(\Delta)}(\pi_{\alpha, \beta}(L+M)).
}
\end{proof}


\begin{thebibliography}{BHKV10}

\bibitem[BKL]{BKL}
G. Bang, T. Kim, and S. Lim. \emph{Singular linear forms over global function fields}. preprint. \href{https://arxiv.org/abs/2404.07752}{arXiv:2404.07752}.

\bibitem[Bou]{Bou}
N. Bourbaki. \emph{Commutative algebra. Chapters 1–7}. Elements of Mathematics (Berlin), Springer-Verlag, Berlin, 1998, Translated from the French, Reprint of the 1989 English translation.

\bibitem[BPP]{BPP}
A. Broise-Alamichel, J. Parkkonen, and F. Paulin. \emph{Equidistribution and counting under equilibrium states in negative curvature and trees. Applications to non-Archimedean Diophantine approximation}. With an Appendix by J. Buzzi. Prog. Math. \textbf{329}, Birkhäuser, 2019.

\bibitem[DFSU]{DFSU}
T. Das, L. Fishman, D. Simmons, and M. Urba\'{n}ski. \emph{A variational principle in the parametric geometry of numbers}. Advances in Mathematics \textbf{437} (2024):109435.

\bibitem[EMM]{EMM}
A. Eskin, G. Margulis, and S. Mozes. \emph{Upper bounds and asymptotics in a quantitative version of the Oppenheim
conjecture}. Ann. of Math. (2), \textbf{147} (1998), no. 1, 93-141.

\bibitem[HN]{HN}
G. Harder and M. S. Narasimhan. \emph{On the cohomology groups of moduli spaces of vector bundles on curves}. Math. Ann. \textbf{212} (1974) 215–248.

\bibitem[KKLM]{KKLM}
S. Kadyrov, D. Kleinbock, E. Lindenstrauss, and G. A. Margulis. \emph{Singular systems of linear forms and
non-escape of mass in the space of lattices}. J. Anal. Math. \textbf{133} (2017) 253–277.

\bibitem[KlST]{KlST} 
D. Kleinbock, R. Shi, and G. Tomanov. \emph{S-adic version of Minkowski’s geometry of numbers and Mahler’s compactness criterion}. J. Numb. Theo. \textbf{174} (2017) 150–163.

\bibitem[MS]{MS}
P. Mercuri and C. Stirpe. \emph{Classification of algebraic function fields with class number one}. J. Number Theory \textbf{154} (2015) 365–374.

\bibitem[Nar]{Nar}
W. Narkiewicz, \emph{Elementary and analytic theory of algebraic numbers}, 3rd Ed., Springer Verlag, 2004.

\bibitem[PR]{PR}
A. Poëls and D. Roy. \emph{Parametric geometry of numbers over a number field and extension of
scalars}. Bull. Soc. Math. France \textbf{151} (2023) 257–303.

\bibitem[Sax]{Sax}
N. de Saxcé. \emph{Non-divergence in the space of lattices}. Groups Geom. Dyn. \textbf{17} (2023), no. 3, 993–1003.

\bibitem[Ser]{Ser}
J-P. Serre, \emph{Trees}, Springer Verlag, 1980, Translation of “Arbres, Amalgames, SL2”, Asterisque \textbf{46}, 1977.

\bibitem[Sol]{Sol}
O. Solan. \emph{Parametric geometry of numbers with general flow}. \href{https://arxiv.org/abs/2106.01707}{arXiv:2106.01707}.
 
\bibitem[SS]{SS}
Q. Shen and S. Shi. \emph{Function fields of class number one}. J. Number Theory \textbf{154} (2015) 375–379.

\bibitem[Wei]{Wei}
A. Weil. \emph{Basic number theory}. Third edition, Springer-Verlag, 1995.
\end{thebibliography}
\end{document}